\documentclass[english,12pt]{amsart}
\usepackage{amssymb}
\usepackage{amsfonts}
\usepackage{mathtools}
\usepackage{amscd}
\usepackage{pb-diagram}
\usepackage[all]{xy}
\usepackage{color}
\usepackage{enumerate}
\usepackage[utf8]{inputenc}
\usepackage{stmaryrd}
\usepackage{todonotes}

\usepackage[T1]{fontenc}
\RequirePackage[colorlinks=true, breaklinks=true, urlcolor= black, linkcolor= black, citecolor= black, bookmarksopen=true,linktocpage=true,plainpages=false,pdfpagelabels]{hyperref}

\CompileMatrices

\swapnumbers

\def\deg{{\rm deg}}
\def\rad{\operatorname{rad}}

\def\radop{\rad_{\mathrm{op}}}
\def\radcl{\rad_{\mathrm{cl}}}
\def\ord{{\rm ord}}

\def\val{{\mathrm{val}}}
\def\Lval{\cL_\val}

\def\LZpan{\cL_{\ZZ_p\langle x \rangle}} 

\def\fine{{\mathrm{fine}}}  

\def\11{{\mathbf 1}}
\def\AA{{\mathbb A}}

\def\CC{{\mathbb C}}

\def\FF{{\mathbb F}}

\def\NN{{\mathbb N}}

\def\PP{{\mathbb P}}
\def\QQ{{\mathbb Q}}
\def\RR{{\mathbb R}}

\def\ZZ{{\mathbb Z}}

\def\cB{{\mathcal B}}

\def\cL{{\mathcal L}}
\def\cM{{\mathcal M}}

\def\cO{{\mathcal O}}

\def\cT{{\mathcal T}}

\mathchardef\alphag="7C0B \mathchardef\betag="7C0C
\mathchardef\gammag="7C0D \mathchardef\deltag="7C0E
\mathchardef\varepsilong="7C22 \mathchardef\varphig="7C27
\mathchardef\psig="7C20 \mathchardef\zetag="7C10
\mathchardef\epsilong="7C0F \mathchardef\rhog="7C1A
\mathchardef\taug="7C1C \mathchardef\upsilong="7C1D
\mathchardef\iotag="7C13 \mathchardef\thetag="7C12
\mathchardef\pig="7C19 \mathchardef\sigmag="7C1B
\mathchardef\etag="7C11 \mathchardef\omegag="7C21
\mathchardef\kappag="7C14 \mathchardef\lambdag="7C15
\mathchardef\mug="7C16 \mathchardef\xig="7C18
\mathchardef\chig="7C1F \mathchardef\nug="7C17
\mathchardef\varthetag="7C23 \mathchardef\varpig="7C24
\mathchardef\varrhog="7C25 \mathchardef\varsigmag="7C26
\mathchardef\Omegag="7C0A \mathchardef\Thetag="7C02
\mathchardef\Sigmag="7C06 \mathchardef\Deltag="7C01
\mathchardef\Phig="7C08 \mathchardef\Gammag="7C00
\mathchardef\Psig="7C09 \mathchardef\Lambdag="7C03
\mathchardef\Xig="7C04 \mathchardef\Pig="7C05
\mathchardef\Upsilong="7C07

\newtheorem{theorem}[subsubsection]{Theorem}
\newtheorem{thm}[subsubsection]{Theorem}
\newtheorem{lem}[subsubsection]{Lemma}
\newtheorem{cor}[subsubsection]{Corollary}
\newtheorem{prop}[subsubsection]{Proposition}

{Addendum}

\theoremstyle{definition}
\newtheorem{defn}[subsubsection]{Definition}

\newtheorem{example}[subsubsection]{Example}

\newtheorem{def-prop}[subsubsection]{Proposition-Definition}
\newtheorem{def-theorem}[subsubsection]{Theorem-Definition}
\newtheorem{def-lem}[subsubsection]{Lemma-Definition}

\theoremstyle{remark}
\newtheorem{remark}[subsubsection]{Remark}

\theoremstyle{plain}

\numberwithin{equation}{subsection}

\def\boxit#1#2{\setbox1=\hbox{\kern#1{#2}\kern#1}%
\dimen1=\ht1 \advance\dimen1 by #1 \dimen2=\dp1 \advance\dimen2 by
#1
\setbox1=\hbox{\vrule height\dimen1 depth\dimen2\box1\vrule}%
\setbox1=\vbox{\hrule\box1\hrule}%
\advance\dimen1 by .4pt \ht1=\dimen1 \advance\dimen2 by .4pt
\dp1=\dimen2 \box1\relax}

\renewcommand{\theequation}{\thesubsection.\arabic{equation}}

\mathchardef\alphag="7C0B \mathchardef\betag="7C0C
\mathchardef\gammag="7C0D \mathchardef\deltag="7C0E
\mathchardef\varepsilong="7C22 \mathchardef\varphig="7C27
\mathchardef\psig="7C20 \mathchardef\zetag="7C10
\mathchardef\epsilong="7C0F \mathchardef\rhog="7C1A
\mathchardef\taug="7C1C \mathchardef\upsilong="7C1D
\mathchardef\iotag="7C13 \mathchardef\thetag="7C12
\mathchardef\pig="7C19 \mathchardef\sigmag="7C1B
\mathchardef\etag="7C11 \mathchardef\omegag="7C21
\mathchardef\kappag="7C14 \mathchardef\lambdag="7C15
\mathchardef\mug="7C16 \mathchardef\xig="7C18
\mathchardef\chig="7C1F \mathchardef\nug="7C17
\mathchardef\varthetag="7C23 \mathchardef\varpig="7C24
\mathchardef\varrhog="7C25 \mathchardef\varsigmag="7C26
\mathchardef\Omegag="7C0A \mathchardef\Thetag="7C02
\mathchardef\Sigmag="7C06 \mathchardef\Deltag="7C01
\mathchardef\Phig="7C08 \mathchardef\Gammag="7C00
\mathchardef\Psig="7C09 \mathchardef\Lambdag="7C03
\mathchardef\Xig="7C04 \mathchardef\Pig="7C05
\mathchardef\Upsilong="7C07

\newcommand{\RV}{\mathrm{RV}}

\newcommand{\rv}{\operatorname{rv}}

\newcommand{\Th}{\operatorname{Th}}

\DeclareMathOperator*{\sgn}{{sgn}}
\DeclareMathOperator*{\LT}{LT}

\newcommand{\trans}{{\mathrm{trans}}}

\newcommand{\cdim}{\#\text{-}\dim}
\newcommand{\saij}{\sum_{j =k_i}^{+\infty} a_{ij} t^j} 

\def\ord{{\rm ord}}

\newcommand{\abs}[1]{\lvert#1\rvert}

\definecolor{immi}{rgb}{0,.6,.1}

\newbox\removebox
\newcommand\remove[1]{%
\setbox\removebox=\ifmmode\hbox{$#1$}\else\hbox{#1}\fi%
\leavevmode
\rlap{\textcolor{blue}{\vrule height0.8ex depth-0.6ex width\wd\removebox}}%
\box\removebox
}
\long\def\bigremove#1{%
\par\setbox\removebox=\vbox{#1}%
\vbox{%
\vbox to0pt{\hbox{\tikz\draw[color=blue,thick] (0,0) -- (\wd\removebox,-\ht\removebox)  (\wd\removebox,0) -- (0,-\ht\removebox);}}
\box\removebox
}
}

\newcommand\acl{\mathrm{acl}}
\newcommand\dcl{\mathrm{dcl}}

\definecolor{orange}{rgb}{1,0.5,0}

\newcommand{\private}[1]{\leavevmode{\scriptsize\color{blue}\marginpar{{\scriptsize Private comment}}#1\par}}
\renewcommand{\private}[1]{}



\thanks{The author V.C-F.\ was partially supported by KU Leuven IF C14/17/083. The author K.H.N. was partially supported by FWO Flanders (Belgium) with grant number 12X3519N and by the Excellence Research Chair “FLCarPA: L-functions in positive characteristic and applications” financed by the Normandy Region. The author M.S.\ was supported by KU Leuven. The author F.V.\ was partially supported by KU Leuven IF C14/17/083, and partially by F.W.O.\ Flanders (Belgium) with grant number 11F1921N}

\title{A Pila--Wilkie theorem for Hensel minimal curves}

\author[Cantoral-Farf\'an]{Victoria Cantoral Farf\'an}
\address{Victoria Cantoral-Farf\'an, Georg-August Universit\"at G\"ottingen, Mathematisches Institut, Bunsenstrasse 3-5, 37073 G\"ottingen, Deutschland}
\email{victoria.cantoralfarfan@mathematik.uni-goettingen.de}
\urladdr{https://sites.google.com/view/victoriacantoral}

\author[Nguyen]{Kien Huu Nguyen}
\address{KU Leuven, Department of Mathematics,
	Celestijnenlaan 200B, B-3001 Leu\-ven, Bel\-gium, and, Normandie Université, Université de Caen Normandie - CNRS, Laboratoire de Mathématiques Nicolas Oresme (LMNO),UMR 6139, 14000 Caen, France}
\email{kien.nguyenhuu@kuleuven.be, huu-kien.nguyen@unicaen.fr}
\urladdr{https://sites.google.com/site/nguyenkienmath/home}

\author[Stout]{Mathias Stout}
\address{Mathias Stout, KU Leuven, Department of Mathematics, B-3001 Leu\-ven, Bel\-gium}
\email{mathias.stout@kuleuven.be}

\author[Vermeulen]{Floris Vermeulen}
\address{Floris Vermeulen, KU Leuven, Department of Mathematics, B-3001 Leu\-ven, Bel\-gium}
\email{floris.vermeulen@kuleuven.be}
\urladdr{https://sites.google.com/view/floris-vermeulen}

\begin{document}

\begin{abstract}
Recently, a new axiomatic framework for tameness in henselian valued fields was developed by Cluckers, Halupczok, Rideau-Kikuchi and Vermeulen and termed Hensel minimality. In this article we develop Diophantine applications of Hensel minimality. We prove a Pila--Wilkie type theorem for transcendental curves definable in Hensel minimal structures. In order to do so, we introduce a new notion of point counting in this context related to dimension counting over the residue field. We examine multiple classes of examples, showcasing the need for this new dimension counting and prove that our bounds are optimal. 
\end{abstract}

\maketitle

\section{Introduction}\label{sec:introduction}

\subsection{The Pila-Wilkie theorem}
In 1989, Bombieri and Pila developed a very fruitful method to count integral and rational points on various types of geometric objects in $\RR^2$~\cite{Bombieri-Pila}. This method is now called the determinant method and is especially well-suited for proving uniform upper bounds on points of bounded height. For a subset $X\subseteq \RR^n$, recall that the counting function is defined as
\[
N(X; B) = \# \{ x \in X\cap \QQ^n \mid H(x)\leq B\},
\]
where $H(a_1/b_1, \ldots, a_n/b_n) = \max(\abs{a_i},\abs{b_i})$ when $\gcd(a_i, b_i) = 1$ for all $i$. For example, if $f: [0,1]\to [0,1]$ is an analytic transcendental function and $X$ denotes its graph, Bombieri and Pila proved that for any $\varepsilon>0$, there is a constant $c_\varepsilon$ such that for all $B > 0$
\[
N(X; B)\leq c_\varepsilon B^\varepsilon.
\]
A vast generalization of this result is the celebrated Pila--Wilkie theorem \cite{PW}. It states that, if $X\subseteq \RR^n$ is definable in an o-minimal structure, then for any $\varepsilon> 0$, there exists a constant $c_\varepsilon$ such that for all $B >0$
\begin{equation} \label{eq:PW}
	N(X^\trans ; B)\leq c_\varepsilon B^\varepsilon.
\end{equation}
Here $X^\trans$ denotes the transcendental part of $X$, obtained from $X$ by removing all positive-dimensional connected semialgebraic subsets of $X$. The proof of this result is based heavily on the existence of $C^r$-parametrizations, which were originally developed by Gromov and Yomdin \cite{YY, YY2, gromov}.

In the non-archimedean setting, such parametrization results were first proved in \cite{CCL-PW}, and a corresponding Pila--Wilkie theorem was obtained for subanalytic sets in $\QQ_p$. These results were further improved in \cite{CFL}, where a uniform version of these bounds was proved for subanalytic sets in $\QQ_p$ and $\FF_p((t))$.

\subsection{Hensel minimality}
Hensel minimality, or h-minimality for short, is a recent framework for tame non-archimedean geometry, developed by  Cluckers, Halupczok and Rideau-Kikuchi in equicharacteristic zero in \cite{CHR} and extended to mixed characteristic together with the fourth author in \cite{CHRV}. It encompasses the aforementioned analytic structure on $\QQ_p$ as a special case, but it applies more broadly, see e.g. \cite[Sec.\,6]{CHR} for several examples. 

Hensel minimality bears a striking resemblance to the classical theory of o-minimality. In an o-minimal structure $K$, each definable subset $X \subseteq K$ is a finite union of intervals and points. In other words, there is some finite tuple $(a_i)_{i \in I}$ such that $X$ is a union of fibres of the map $x \mapsto (\sgn(x-a_i))_{i \in I}$. Roughly speaking, h-minimality replaces the sign map by the \emph{leading term} map $\rv \colon K \to K^{\times}/(1 +\cM_K) \cup \{0\}$, where $K$ is a valued field and $\cM_K$ is the maximal ideal of its valuation ring. 

The goal of this article is to develop an analogue of the Pila--Wilkie theorem in an h-minimal context. For this purpose, we need two important consequences of Hensel minimality: a cell decomposition statement (Theorem \ref{thm:celldecomp.1.lipschitz}) and the Jacobian property (Theorem \ref{thm: jacobian property}). These theorems are analogues of o-minimal cell decomposition and the monotonicity theorem, respectively. We use them to prove a \emph{$T_r$-parametrization} statement for curves definable in Hensel minimal structures (Theorem \ref{thm:Tr.param.curves}). These $T_r$-parametrizations are analogues of the $C^r$-parametrizations used in the proof of the o-minimal Pila-Wilkie theorem, and form a key technical ingredient.

\subsection{Counting dimension}
Let $k$ be a field of characteristic zero and denote by $k((t))$ the field of Laurent series over $k$. For any natural number $s$, we let $k[t]_s \subseteq k((t))$ be the set of polynomials in $t$ of degree less than $s$. For (transcendental definable) curves $X \subseteq k((t))^n$ we study the number of points on $X_s \coloneqq X \cap (k[t]_s)^n$ as a valued field analogue of rational points of bounded height. In particular, we are interested in bounding the growth in terms of $s$, similar to the Pila-Wilkie theorem.

When $k = \CC$ and $X$ is a transcendental curve, definable in the subanalytic structure on $\CC((t))$, then $X_s$ is finite, for each $s \in \NN$~\cite[Thm.\,1]{BCN}. 
However, this result does not carry over to arbitrary h-minimal structures and in Section \ref{sec:examples} we give explicit examples of transcendental curves $X$, definable in some h-minimal structure, with infinite $X_s$. This issue is resolved by counting relative to the residue field. More precisely, we introduce the notion of \emph{counting dimension}, denoted by $\cdim$, and consider bounds of the form
\begin{equation} \label{eq:cdim}
\cdim(X_s)  \leq (N(s),d,e(s)),
\end{equation}
where $N$ and $e$ are functions $\NN \to \NN$ and $d \in \NN$ is constant.
Intuitively, the above inequality can be thought of as a means of expressing that
\[
\# X_s \leq N(s) \# (k[t]/(t^{e(s)}))^d,
\] 
even when $k$ is infinite. When $N = N(s)$ is constant, $d \cdot e(s)$ can be thought of as bounding the growth of the $k$-dimension of $X_s$. 

For $k = \FF_p$, note that the right-hand side becomes $ N(s) p^{d e(s)}$. Comparing this to the known results for transcendental definable curves in $\FF_p((t))$ (\cite[Thm.\,B]{CFL}), leads to the following question for transcendental curves $X \subseteq k((t))^n$: if $\varepsilon > 0$ is given, can we take $N = N_{\varepsilon}$ constant, $d =1$ and $e(s) = \lceil \varepsilon \cdot s \rceil$ in (\ref{eq:cdim})? 
We stress that the importance of the counting dimension is that it makes this question meaningful when $k$ is infinite. 

We will return to the motivations for the counting dimension, after precisely defining it for any henselian valued field $K$ (and not just $k((t))$). Theorem \ref{thm:CC-c-dim} then positively answers our question above: if $K$ is an equicharacteristic zero henselian valued field which is h-minimal, satisfying some mild extra conditions, then for any transcendental definable curve $X \subseteq K^n$ and any $\varepsilon > 0$ there exists some constant $N_{\varepsilon}>0 $ such that
\[
\cdim(X_s) \leq (N_{\varepsilon},1, \lceil \varepsilon \cdot s \rceil).
\]
We moreover show that this bound is optimal, by constructing certain transcendental definable curves of a specific form.

Additionally, we consider the case of algebraic curves in Theorem \ref{thm:alg.curves}, and prove the analogue of the classical Bombieri--Pila bound here. Let us also mention that the main obstacle in extending these results to higher dimensions is that these parametrization results are only known for curves under Hensel minimality. Especially the higher-dimensional geometry under Hensel minimality has to develop further.


\subsection*{Acknowledgements} The authors would like to thank Raf Cluckers for suggesting this topic and the various discussions surrounding it. The authors also thank the anonymous referee for the many helpful comments which improved this article.

\section{The counting dimension}\label{sec: main results}

\subsection{Counting in valued fields}\label{sec: counting in valued fields}

In this section we introduce the counting dimension and prove some basic results about it. We then state our main results on the counting dimension of transcendental and algebraic curves. 

Let $K$ be a non-archimedean valued field equipped with an $\cL$-structure, for some language $\cL$ expanding the language of valued fields $\cL_\val = \{0,1, +, \cdot, \cO_K\}$. For $A\subset K$, a set $X\subset K^n$ is called $A$-definable if it is definable in $\cL$, using parameters only from $A$. We call a set definable if it is $K$-definable. We denote by $k$ the residue field of $K$, by $\cO_K$ the valuation ring and by $\cM_K$ the maximal ideal. Let $\Gamma_K^\times$ be the value group, where the valuation is written multiplicatively $|\cdot|: K\to \Gamma_K$.

Assume that $K$ is henselian of equicharacteristic zero. Then there always exists a lift $\tilde{k}\subset K$ of the residue field $k$, i.e.\ a subfield of $K$ which maps bijectively to $k$ under the reduction map $\cO_K\to k$. Fix also a \emph{pseudo-uniformizer} $t\in K$, recall that this is any non-zero element of $K$ with $|t|<1$. For a positive integer $s$, define
\[
\tilde{k}[t]_s = \left\{\sum_{i=0}^{s-1} a_it^i\mid a_i\in \tilde{k}\right\}.
\]
If $X$ is a subset of $K^n$ we define $X_s$ to be $X\cap \tilde{k}[t]_s^n$. We call $X_s$ the \emph{set of rational points of height at most $s$ on $X$}. Note that this set depends on the choice of pseudo-uniformizer $t$ and on the lift $\tilde{k}$. 

The prototypical example to keep in mind is $K=k((t))$ for some characteristic zero field $k$, with $\tilde{k}=k$ and $t$ as pseudo-uniformizer. Here $\tilde{k}[t]_s$ is simply the set of polynomials over $k$ of degree at most $s-1$. 

The set $X_s$ is considered as a suitable analogue for the set of rational points of bounded height on $X$, where the height is captured by $s$. We will be interested in bounding the size of $X_s$, as $s$ grows, for various types of subsets of $K^n$. We cannot simply use the number of points on $X_s$ as a measure of size, since this set will typically be infinite. Instead we introduce the \emph{counting dimension}, to measure the size of $X_s$ relative to the residue field.

\begin{defn}\label{def:counting.dim}
Let $K$ be a henselian valued field of equicharacteristic zero equipped with an $\cL$-structure, for some language $\cL$ expanding the language of valued fields. Fix a pseudo-uniformizer $t$ of $K$ and a lift $\tilde{k}$ of the residue field. Let $X$ be a subset of $K^n$, let $d$ be a positive integer and let $N, e: \NN\to \NN$ be functions. Then we say that $X$ has \emph{counting dimension bounded by $(N, d, e)$} if there exists a definable function $f: X\to \cO_K^d$ such that for every positive integer $s$, the composition
\[
X_s\xrightarrow{f} \cO_K^d \xrightarrow{\operatorname{proj}} \left( \frac{\cO_K}{(t^{e(s)})} \right)^d
\]
has finite fibres of size at most $N(s)$. Here $\operatorname{proj}$ is the componentwise reduction map modulo $t^{e(s)}$. 

We will use the notation
\[
\cdim X_s \leq (N(s), d, e(s) )
\]
to mean that the counting dimension of $X$ is bounded by $(N,d,e)$.
\end{defn}

This definition depends on the choice of pseudo-uniformizer $t$, the lift $\tilde{k}$ and on the language $\cL$. However, we suppress these in notation and always assume a fixed choice of $t, \tilde{k}$ and $\cL$. 

Our definition of the counting dimension is motivated on the one hand by counting rational points on definable subsets in local fields of mixed characteristic as in \cite{CCL-PW, CFL, CHRV}. In that case, one has to use a different notion of rational points of bounded height, as there is no lift of the residue field. Consider for example $K=\QQ_p$. Then, if $X$ is a subset of $\QQ_p^n$, define
\[
X_s = \{x\in X\cap \ZZ^n\mid 0\leq x_i\leq s \text{ for all } i \}.
\]
Using this definition of rational points, if $X$ has counting dimension bounded by $(N(s), d, e(s))$, then for every $s$, $X_s$ will contain no more than $N(s)p^{de(s)}$ points. So a bound on the counting dimension gives a corresponding bound on the number of points in $X_s$.

Our second motivation comes from the relation of the counting dimension to the Zariski dimension over the residue field, as in \cite[Sec.\,5]{CCL-PW}. Let $K=\CC((t))$ and let $X$ be a subset of $K^n$. Then for every $s$, $X_s$ is a subset of $\CC[t]_s$, which can be naturally identified with $\CC^{ns}$. If $X$ is algebraic, then $X_s$ will be a constructible set in $\CC^{ns}$, and one can wonder how the Zariski dimension of $X_s$ grows with $s$. In \cite[Sec.\,5]{CCL-PW}, a bound for this quantity is provided. We obtain a similar bound using the counting dimension instead of the Zariski dimension.

\subsection{Main results}

Our main results concern the counting dimension of algebraic and transcendental curves definable in h-minimal structures.

By a \emph{curve $C\subset K^n$} we mean a set for which there exists a linear map $p:K^n\to K$ such that $p(C)$ is infinite and such that $p$ has finite fibers on $C$. If the theory of $K$ in $\cL$ is 1-h-minimal, see \cite[Def.\,2.3.3]{CHR} or section \ref{sec: henselminimality} below, then by dimension theory \cite[Thm.\,5.3.4]{CHR}, a definable curve in $K^n$ is the same as a definable set of dimension 1. We call a curve $C\subset K^n$ \emph{transcendental} if every algebraic curve in $K^n$ has finite intersection with $C$. 

Let $K$ be a henselian valued field equipped with an $\cL$-structure and assume that $\Th_\cL(K)$ is 1-h-minimal. Then we say that $\acl=\dcl$ for $K$ if algebraic Skolem functions exist in every model of $\Th_\cL(K)$. By this we mean that for any model $K'$ of $\Th_\cL(K)$ and any subset $A\subset K'$ we have that $\acl_{K'}(A) = \dcl_{K'}(A)$. 

Our main result is the following analogue of the Pila--Wilkie theorem on the counting dimension of transcendental curves definable in Hensel minimal structures.

\begin{thm}\label{thm:CC-c-dim}
Suppose that $K$ is a henselian valued field of equicharacteristic $0$ equipped with a 1-h-minimal structure. Fix a pseudo-uniformizer $t$ and a lift of the residue field $\tilde{k}$. Suppose that $\acl=\dcl$ in $K$ and that the subgroup of $b$-th powers in $k^\times$ has finite index, for some integer $b>1$. Let $C\subset \cO_K^n$ be a transcendental definable curve. Then for each $\varepsilon>0$ there is a constant $N$ such that for each integer $s\geq 0$
\[
\#\text{-}\dim (C_s) \leq (N,1,\lceil \varepsilon\cdot s \rceil).
\]  
Furthermore, the constant $N$ can be taken to hold uniformly throughout all transcendental members of a given definable family of definable curves. 
\end{thm}

The key aspect here is the slow growth of the last component of the counting dimension, similar to the Pila--Wilkie theorem in the o-minimal setting \cite[Thm.\,1.10]{PW}. The strategy of the proof is as follows. We use the notion of $T_r$-parametrizations, which form a suitable analogue for $C^r$-parametrizations in the non-archimedean setting, see e.g.\ ~\cite{CCL-PW, CFL}. 
\begin{enumerate}
\item We apply cell decomposition to find a $T_1$-parametrization of $C$. The existence of such a cell decomposition follows from 1-h-minimality under the extra assumption that $\acl=\dcl$ in $K$, see \cite[Thm.\,5.2.4 Add.\,5]{CHR} or section \ref{sec: henselminimality} below.
\item Using substitutions of the form $x\mapsto x^r$, we may even assume that we have a $T_r$-parametrization, for some suitably chosen integer $r$. For this, we need that the subgroup of $b$-th powers of $k^\times$ has a finite index in $k^\times$. 
\item We then use an adaptation of the Bombieri--Pila determinant method to catch all rational points of bounded height in a small ball in a single hypersurface. 
\item Finally, the fact that $C$ is transcendental and definable in a 1-h-minimal structure then gives the desired result. Indeed, this will follow from uniform finiteness in definable families, see \cite[Lem.\,2.5.3]{CHR} or section \ref{sec: henselminimality} below.
\end{enumerate}
The crucial ingredient to extend Theorem~\ref{thm:CC-c-dim} to higher-dimensional transcendental sets is the existence of $T_r$-parametrizations, which have not been proven to exist in general 1-h-minimal structures. However, if one assumes the existence of these parametrizations, then Theorem~\ref{thm:CC-c-dim} follows via a similar approach based on the determinant method.

A mixed characteristic analogue of this result in $\QQ_p$ was proven by Cluckers--Halupczok--Rideau-Kikuchi--Vermeulen~\cite[Thm.\,4.1.6]{CHRV}. Here the notion of rational points of bounded height is defined as above. Namely, for $X$ a subset of $\QQ_p^n$ define 
\[
X_s = \{x\in X\cap \ZZ^n\mid 0\leq x_i\leq s \text{ for all } i \}.
\]
Then~\cite[Thm.\,4.1.6]{CHRV} states that if $\QQ_p$ carries a $1$-h-minimal structure with $\acl=\dcl$, and if $C\subset \QQ_p^n$ is a transcendental definable curve, then for each $\varepsilon>0$ there is a constant $c$ such that for all $H\geq 1$ we have
\[
\# C_s \leq cs^\varepsilon.
\]


In this article, we will restrict to equicharacteristic zero. The methods for both proofs are quite similar, one major difference being that the residue field is no longer finite. This is the reason for introducing the counting dimension. Restricting to equicharacteristic zero has the added benefit that Hensel minimality is slightly easier to work with. 

We also prove that Theorem~\ref{thm:CC-c-dim} is optimal, in the sense that last component $\lceil \varepsilon \cdot s \rceil$ of our bound cannot be improved. In more detail, one cannot replace it by a sublinear function $e(s)$, even if $N(s)$ is allowed to be completely arbitrary. 

\begin{thm} \label{thm:optimal}	
Let $k$ be a field of characteristic zero. There exists a $1$-h-minimal structure on $k((t))$ with $\acl=\dcl$ satisfying the following. Given a sublinear function $e \colon \NN \to \NN$ and any $N \colon \NN \to \NN$, there exists a definable transcendental curve $C \subseteq k((t))^2$ such that its counting dimension is not bounded by $(N(s),1,e(s))$.
\end{thm}

In contrast with this result, we consider in Section~\ref{sec:QQpt} a specific analytic structure on $\QQ_p((t))$ for which we are able to prove that the counting dimension of every definable transcendental curve is bounded by $(N,1,1)$ for some integer $N>0$. This structure already contains many interesting examples of transcendental definable curves. For example, the graph of the exponential function $\exp: p\ZZ_p+t\QQ_p[[t]]\to \QQ_p((t))$ is definable.

As for algebraic curves, we prove the following theorem, generalizing the results from \cite[Sec.\,5]{CCL-PW}.

\begin{thm}\label{thm:alg.curves}
Let $K$ be a henselian valued field of equicharacteristic $0$ equipped with a 1-h-minimal structure. Assume that $\acl=\dcl$ and that the subgroup of $b$-th powers in $k^\times$ has finite index, for some $b>1$. Let $C\subset K^2$ be an irreducible algebraic curve of degree $d$, for some positive integer $d$. Then there exists a constant $c_d$, depending only on $d$, such that
\[
\cdim C_s \leq (c_d s, 1, \lceil s/d\rceil).
\]
\end{thm}

This theorem can be considered the analogue of the classical Bombieri--Pila theorem \cite{Bombieri-Pila}. By considering the example $y = x^d$, we will show that one cannot improve the last component of the counting dimension in this result.

\subsection{Some basic properties}

Let us list some basic properties of the counting dimension. We will often use these implicitly in our proofs.

\begin{prop}\label{prop:basic.props}
Let $K$ be a non-archimedean valued field in some language $\cL$ expanding the language of valued fields and let $X, X'$ be definable subsets of $K^n$. Assume that 
\[
\#\text{-}\dim (X_s) \leq (N(s),d,e(s)), \quad \#\text{-}\dim(X'_s) \leq (N'(s),d',e'(s))
\]
for some integers $d, d'$ and some functions $N,N',e,e': \NN\to \NN$. Then
\begin{enumerate}
\item $\#\text{-}\dim ( (X\cup X')_s ) \leq (N(s)+N'(s), \max\{d,d'\}, \max\{e(s),e'(s)\})$,
\item $\#\text{-}\dim ( (X\times X')_s ) \leq (N(s)N'(s), d + d', \max\{e(s),e'(s)\})$,
\item if $f: X'\to X$ is a definable map with finite fibres of size at most $N''$, for some integer $N''$, then the counting dimension of $X'$ is bounded by $(N'' N, d, e)$. If moreover $f$ is surjective and $\acl=\dcl$ in $K$, then the counting dimension of $X$ is bounded by $(N', d', e')$.
\end{enumerate}
\end{prop}
\begin{proof}
The proof of the first two properties is straightforward, so let us prove the last property. Let $g: X\to \cO_K^{d}$ be a definable map such that for every positive integer $s$, the composition
\[
X_s\xrightarrow{g} \cO_K^{d} \xrightarrow{\operatorname{proj}} \left( \frac{\cO_K}{(t^{e(s)})} \right)^{d}
\]
has finite fibres of size at most $N(s)$ and take a similar such definable map $g': X'\to \cO_K^{d'}$ for $X'$. Then for every positive integer $s$, the composition
\[
\operatorname{proj}\circ g \circ f: X'_s\to \left( \frac{\cO_K}{(t^{e(s)})} \right)^{d}
\]
has finite fibres of size at most $N''N(s)$, so the counting dimension of $X'$ is bounded by $(N''N, d, e)$.

Conversely, using $\acl=\dcl$ there exists a section $f': X\to X'$ for $f$. Then for every positive integer $s$, the composition
\[
\operatorname{proj}\circ g' \circ f': X_s\to \left( \frac{\cO_K}{(t^{e'(s)})} \right)^{d'}
\]
has finite fibres of size at most $N'(s)$.
\end{proof}

\subsection{Some examples} \label{sec:examples}

We give some examples of bounding the counting dimension of transcendental curves, which the reader can keep in mind throughout this article. In each of these, we consider a curve $C$ in a valued field $K$ which is definable in some 1-h-minimal structure on $K$. We then give an upper bound for the counting dimension of $C$ by simply computing the sets $C_s$.

\begin{example} \label{ex:p-adic}
Consider the valued field $K = \QQ_p((t))$ with valuation ring $\QQ_p[[t]]$. In Section~\ref{sec:QQpt} we argue that there is a 1-h-minimal structure on $K$ in which the exponential map 
\[
\exp: p\ZZ_p + t\QQ_p[[t]]\to \QQ_p((t)): z\mapsto \sum_{i\geq 0} \frac{z^i}{i!}
\]
is definable. Let us write $U = p\ZZ_p + t\QQ_p[[t]]$.

Let $C$ be the graph of this exponential function. We claim that this is a transcendental set. Indeed, suppose that $f(x,y)\in K[x,y]$ is a non-zero polynomial such that $f(x,\exp x) = 0$ for infinitely many $x\in U$. We take such an $f$ of minimal degree. By h-minimality \cite[Lem.\,2.5.2]{CHR}, there is then an open ball $B\subset U$ on which this holds. Moreover, $\exp$ is differentiable on $U$ with derivative $\exp$. Write $f(x,y) = \sum_{i=0}^d f_i(x)y^i$. Define the polynomial
\[
g(x,y) = \sum_{i=0}^d f_i'(x)y^i + \sum_{i=0}^{d-1} (i-d)f_i(x)y^i.
\]
This polynomial is non-zero and  has degree strictly smaller than $f$ and for $x\in B$ we have
\[
g(x, \exp x) = \dfrac{\mathrm{d}}{ \mathrm{d} x} (f(x, \exp x)) - df(x, \exp x) = 0.
\]
This is the desired contradiction, showing that $C$ is a transcendental set.

To compute the counting dimension, we use $t$ as a uniformizer and $\QQ_p\subset \QQ_p((t))$ as a lift of the residue field. Now, if $x$ is in $p\ZZ_p$ then $\exp x$ is in $\QQ_p$. Hence
\[
C_1 = \{(x, \exp x)\mid x\in p\ZZ_p\}.
\]
In particular, $C_1$ is infinite. We claim that the counting dimension of $C$ is bounded by $(1,1,1)$. For this purpose, consider the map $C\to \cO_K: (x, \exp x)\mapsto x$ followed by projection $\cO_K\to \cO_K/\cM_K = \QQ_p$. This clearly has finite fibres on $C_1$. Even more, if $x$ is an element of $U\cap \QQ_p[t]$ which is not in $p\ZZ_p$ then automatically $\exp x$ is not in $\QQ_p[t]$. Thus for any positive integer $s$, $C_1=C_s$ and the counting dimension of $C$ is bounded by $(1,1,1)$.
\end{example}

\begin{example}
Consider the field $\RR((t))$ in the language of ordered valued fields. We expand the language by the full Weierstrass system $\cB$ as in~\cite[Sec.\,3.1]{CLip}. In more detail, let
\[
A_{n,\alpha}((\ZZ))=\{\sum_{i\in I}f_it^i \mid f_i\in A_{n,\alpha}, I\subset \ZZ \textnormal{ well ordered}\},
\]
where $A_{n, \alpha}$ is the ring of real power series in $\RR[[\xi_1,...,\xi_n]]$ with radius of convergence $>\alpha$, and define $B_{n, \alpha} = A_{n, \alpha}(\ZZ)$. This is a real Weierstrass system in the sense of~\cite[Def.\,3.1.1]{CLip}, and we equip $\RR((t))$ with real analytic $\cB$-structure. In particular, we have function symbols for all elements of the Weierstrass system $\cB$. In \cite{NSV24}, it is shown that the theory of $\RR((t))$ in this language is $1$-h-minimal. Now, by our choice of $\cB$ the exponential
\[
\exp: (-1, 1) + t\RR[[t]]\to \RR((t)): z\mapsto \sum_{i\geq 0} \frac{z^i}{i!}
\]
is definable. Denote by $C$ the graph of $\exp$, which as above is a transcendental set. We use $\RR\subset \RR((t))$ as a lift of the residue field and $t$ as our choice of uniformizer. Then $C_1$ is infinite, since if $x$ is in $(-1, 1)$ then $\exp(x)$ is again real. Consider the reduction map 
\[
C\to \cO_K/(t): (x, \exp x) \mapsto x \bmod t.
\]
Then this has finite fibres of cardinality at most $1$ above $C_1$. In fact, if $x$ is in $\RR[t]_s$ but not in $\RR$ then $\exp(x)$ is never in $\RR[t]$. Thus for any $s\geq 1$ we have $C_s = C_1$ and so the counting dimension of $C$ is bounded by $(1,1,1)$.
\end{example}

\begin{example}
Denote by $\cL_\mathrm{omin}$ the language of ordered rings. For any real number $r>0$, let $f_r$ denote the function $\RR_{>0}\to \RR_{>0}: x\mapsto x^r$ and consider the expansion $\cL$ of $\cL_\mathrm{omin}$ where we have a function symbol for every $f_r$. The theory of $\RR$ in $\cL$ is o-minimal, since it is a reduct of $\RR_{\exp}$. Let $K$ be a proper elementary extension of $\RR$. Then we may turn $K$ into a valued field by taking for $\cO_K$ the convex closure of $\RR$. Note that the theory $\cT$ is power-bounded, so by \cite[Thm.\, 7.2.4]{CHR} the theory of $K$ in the language $\cL\cup\{\cO_K\}$ is 1-h-minimal. Let $t$ be a pseudo-uniformizer of $K$ and denote by $\tilde{k}$ a lift of the residue field. Then $\tilde{k}$ is again a real closed field. For example, $K$ might be the field of Hahn series $\RR((t^\RR))$ with pseudo-uniformizer $t$ and $\tilde{k}=\RR$.

Now let $C$ be the graph of the function $K_{>0}\to K_{>0}: x\mapsto x^\pi$. This is a definable set by our choice of language. Clearly, this is also a transcendental set. But $C_1$ is simply the graph of $x\mapsto x^\pi$ on $\tilde{k}$, which certainly contains this graph on $\RR$. In particular, this set is infinite. Also note that $C_s = C_1$ for any $s\geq 1$ so that the counting dimension of $C$ is bounded by $(1,1,1)$.
\end{example}

\section{Notation and background}\label{sec:notations and background}

\subsection{Hensel minimality}\label{sec: henselminimality}

In this section, we record some background material on Hensel minimality. We refer to \cite{CHR} and \cite{CHRV} for further details.

Let $\cL$ be a language containing the language of valued fields $\cL_{\mathrm{val}} = \{0, 1, +, \cdot, \cO_K\}$. Let $\cT$ be a complete $\cL$-theory whose models are non-trivially valued fields of equi-characteristic zero. Let $K$ be a model of $\cT$. We denote by $\cO_K$ the valuation ring of $K$ and by $\Gamma_K^\times$ the valuation group. The valuation will be denoted by $|\cdot |: K\to \Gamma_K = \Gamma_K^\times \cup\{0\}$. By an open ball we will mean a set of the form $B_{<\lambda}(a) = \{x\in K\mid |x-a|<\lambda\}$, where $a\in K$, $\lambda\in \Gamma_K^\times$. Similarly, a closed ball is a set of the form $B_{\leq \lambda}(a) = \{x\in K\mid |x-a|\leq \lambda\}$. If $B$ is an open ball as above, we denote by $\radop B$ its radius $\lambda$, and similarly we use $\radcl B$ for $\lambda$ if $B$ is closed ball.

For $\lambda\leq 1$ an element of $\Gamma_K^\times$ let $I_\lambda$ be the ideal $\{x\in K\mid |x|<\lambda\}$. We define $\RV_\lambda^\times$ to be $K^\times/(1+I_\lambda)$, with quotient map
\[
\rv_\lambda: K^\times \to \RV_\lambda^\times.
\]
We also consider $\RV_\lambda = \RV_\lambda^\times\cup\{0\}$. The map $\rv_\lambda$ extends to $K\to \RV_\lambda$ via $\rv_\lambda(0)=0$. We will write $\RV = \RV_1$ and $\rv=\rv_1$. The set $\RV$ combines information from the residue field and the value group. Indeed, there is a short exact sequence 
\[
1\to \left (\cO_K/\cM_K\right)^\times \to \RV^\times \to \Gamma_K^\times \to 1.
\]

Now let $\lambda\leq 1$ be in $\Gamma_K^\times$ and let $X$ be a subset of $K$. We say that a finite set $C$ \emph{$\lambda$-prepares $X$} if the following holds: for any $x, y\in K$, if 
\[
\rv_\lambda(x-c) = \rv_\lambda(y-c), \text{ for all } c\in C
\]
then either $x$ and $y$ are both in $X$, or they are both not in $X$. If $(\xi_c)_c\in \RV_\lambda^{\# C}$ then the set
\[
\{x\in K\mid \rv_\lambda(x-c) = \xi_c \text{ for all }c\in C \}
\]
is said to be a ball \emph{$\lambda$-next to $C$} (if it is disjoint from $C$). Note that if such a set is disjoint from $C$, then it is indeed an open ball. We can rephrase preparing as follows. A finite set $C$ $\lambda$-prepares $X$ if for any ball $B$ $\lambda$-next to $C$, either $B\subseteq X$ or $B\cap X = \emptyset$. Note also that the balls $1$-next to a finite set $C$ are precisely the maximal open balls disjoint from $C$.

\subsection{Consequences of Hensel minimality}\label{sec: consequences hminimality}

For this section, we fix a field $K$ of equicharacteristic zero equipped with an $\cL$-structure which is $1$-h-minimal. By \cite{CHR}, we may freely add constants from $K$ to the language $\cL$ and preserve $1$-h-minimality. Many of the results below are formulated only for $\emptyset$-definable objects, but therefore hold just as well for $A$-definable objects, for $A\subset K$.

Hensel minimality implies tameness results on various definable objects. For functions there is the Jacobian property and Taylor approximation.

\begin{thm}[{{\cite[Cor.\,3.2.6]{CHR}, Jacobian property}}]\label{thm: jacobian property}
Let $f: K\to K$ be a $\emptyset$-definable function. Then there exists a finite $\emptyset$-definable set $C$ such that for every $\lambda\leq 1$ in $\Gamma_K^\times$, every ball $B$ $\lambda$-next to $C$ and every $x_0, x\in B$, $x\neq x_0$, we have that
\begin{enumerate}
\item the derivative $f'$ (as defined in the usual way) exists on $B$ and $\rv_\lambda\circ f'$ is constant on $B$,
\item $\rv_\lambda((f(x)-f(x_0)) / (x-x_0) ) = \rv_\lambda(f')$,
\item for any open ball $B'\subset B$, $f(B')$ is either a point or an open ball.
\end{enumerate}
\end{thm}

Note in particular that (1) implies that $|f'|$ is constant on balls 1-next to $C$, since $\Gamma_K$ is a quotient of $\RV_\lambda$. We will also use the following corollary. Recall that $\tilde{k}\subset K$ is a lift of the residue field.

\begin{cor} \label{cor:jac}
Let $\tilde{k}\subset K$ be a lift of the residue field of $K$. Let $f \colon \cO_K \to \cO_K$ be a $\emptyset$-definable function. Then, for all but finitely many $a \in \tilde{k}$ the following property holds: for all $x, x_0 \in a + \cM_K$
\[
\abs{f(x) - f(x_0)} \leq \abs{x - x_0}.
\]
\end{cor}
\begin{proof}
By the Jacobian property there exists a finite $\emptyset$-definable set $C \subseteq \cO_K$ such that for every ball $B$ which is 1-next to $C$ there is a $\mu_B \in \Gamma_K$ such that for all $x,x_0 \in B$
\[
\abs{f(x) - f(x_0)} = \mu_B \abs{x - x_0}.
\]
Moreover, by \cite[Cor.\,3.1.6]{CHR} we may additionally assume that if such a $B$ is open, then $f(B)$ is an open ball of radius $\mu_B \radop(B)$.
As $C$ is finite and the balls $a + \cM_K$ with $a \in \tilde{k}$ are pairwise disjoint, it follows that only finitely many $a + \cM_K$ contain a point of $C$. Thus all but finitely many $a + \cM_K$ are $1$-next to $C$. 
Now suppose  $\mu_B > 1$ for some $B = a + \cM_K$ which is disjoint from $C$. This implies that $f(a + \cM_K) = \cO_K$ (and this is only possible if the value group is discrete). In particular we can only have $\mu_B > 1$ for finitely many such $B$. Indeed, else $f^{-1}(y)$ would be infinite for all $y \in \cO_K$, contradicting \cite[Lem.\,2.8.1]{CHR}. Hence, the desired property holds for cofinitely many $a \in \tilde{k}$.
\end{proof}

The second result we need is about Taylor approximation. For a function $f: X\subset K\to K$ on an open set $X$ which is $r$-fold differentiable and $x_0\in X$ we define the \emph{$r$-th order Taylor polynomial of $f$ at $x_0$} to be as usual
\[
T_{f, x_0}^{\leq r}(x) = T_{f, x_0}^{<r+1}(x) = \sum_{i=0}^r \frac{f^{(i)}(x_0)}{i!}(x-x_0)^i.
\]
The following result basically states that any definable function can be well approximated by its Taylor polynomial up to some fixed order, at least away from finitely many points.

\begin{thm}[{{\cite[Thm.\,3.2.2]{CHR}, Taylor approximation of order $r$}}]\label{thm: taylor approximation}
Let $f: K\to K$ be a $\emptyset$-definable function and fix a positive integer $r$. Then there exists a finite $\emptyset$-definable set $C$ such that for every ball $B$ 1-next to $C$, $f$ is $(r+1)$-fold differentiable on $B$, $|f^{(r+1)}|$ is constant on $B$ and for $x, x_0\in B$ we have
\[
|f(x)-T_{f, x_0}^{\leq r}(x)| \leq |f^{(r+1)}(x_0)(x-x_0)^{r+1}|.
\]
\end{thm}

We will also need results on cell decomposition. If $\cT$ is a 1-h-minimal theory, then by \cite[Prop.\,4.3.3]{CHR} there is an expansion of the language by predicates on cartesian powers of $\RV$ such that the resulting structure is still 1-h-minimal and we have $\acl=\dcl$. In particular, we can typically assume that $\acl=\dcl$ without any problems. 

\begin{defn}
Let $A\subset K$ be a parameter set. For $n\geq m$, denote by $\pi_{\leq m}: K^n\to K^m$ the projection on the first $m$ coordinates. Let $X\subset K^n$ be an $A$-definable set. Consider, for $i=1, ..., n$, values $j_i\in \{0,1\}$ and $A$-definable functions $c_i: \pi_{<i}(X)\to K$. Fix also an $A$-definable set
\[
R\subseteq \prod_{i=1}^n (j_i\cdot \RV^\times),
\]
where $0\cdot \RV^\times = \{0\}$. We say that $X$ is an $A$-definable cell if 
\[
X = \{x\in K^n \mid \rv(x_i-c_i(\pi_{<i}(x)))_{i=1,...,n}\in R \}.
\]
We call $X$ a cell of type $(j_1, ..., j_n)$. The functions $c_i$ are called the \emph{cell centers}. A \emph{twisted box} of the cell $X$ is a set of the form
\[
\{x\in K^n \mid \rv(x_i-c_i(\pi_{<i}(x)))_{i=1,...,n} = r \},
\]
for $r\in R$.
\end{defn}

By \cite[Thm.\,5.2.4]{CHR}, for a $\emptyset$-definable set $X\subseteq K^n$ there always exists a $\emptyset$-definable \emph{cell decomposition}, i.e.\ a partition of $X$ into finitely many $\emptyset$-definable cells $A_\ell$. We will need the following variant. Recall that a function $f: X\subset K^n\to K^m$ is said to be \emph{1-Lipschitz} if for all $x, x'\in X$ we have
\[
|f(x)-f(x')| \leq |x-x'|,
\]
where we use the maximum norm on $K^n$.

\begin{thm}[{{\cite[Thm.\,5.2.4, Add.\,5]{CHR}, Cell decomposition}}]\label{thm:celldecomp.1.lipschitz}
Assume that $K$ carries a 1-h-minimal structure with $\acl=\dcl$. Let $X\subset K^n$ be $\emptyset$-definable. Then there exist a partition of $X$ into finitely many $\emptyset$-definable sets $A_\ell$ such that for every $\ell$ there is some coordinate permutation $\sigma_\ell: K^n\to K^n$ such that $\sigma_\ell(A_\ell)$ is a cell of type $(1, ..., 1, 0, ..., 0)$ and such that each component of each center is 1-Lipschitz.
\end{thm}

\subsection{$T_r$-approximation}\label{sec:Trapprox}

We recall some useful definitions and results from \cite[\S 4.2]{CFL} about $T_r$-approximation. Let $K$ be a henselian valued field of equicharacteristic zero which is 1-h-minimal in some language $\cL$ expanding the language of valued fields.

\begin{defn}\label{def: T_r-approximation}
Let $U\subseteq K^m$ be an open set, let $\psi=(\psi_1, ...,\psi_n):U\to \mathcal{O}_K^n$ be a function, and let $r>0$ be an integer. We say that $\psi$ satisfies \emph{$T_r$-approximation} if for each $y\in U$ there is an $n$-tuple $T_{y}^{<r}$ of polynomials with coefficients in $\mathcal{O}_K$ and of degree less than $r$ that satisfies
\begin{equation}
|\psi(x)-T_{y}^{<r}(x)|\leq |x-y|^r \qquad \text{for all }\; x\in U.
\end{equation}
Let $X$ be a definable subset of $\cO_K^n$ of dimension $m$. We say that a definable family $(\varphi_i)_{i\in I}$ of functions $\varphi_i:U_i\to X_i\subseteq \mathcal{O}_K^n$ is a \emph{$T_r$-parametrization of $X$} if $X=\bigcup_{i\in I}X_i$ and each $\varphi_i$ is surjective and satisfies $T_r$-approximation.
\end{defn}

\begin{defn}\label{def: Notations}
Let $\alpha=(\alpha_1,...,\alpha_m)\in \NN^m$ and define $|\alpha|=\alpha_1+\cdots +\alpha_m$. We define the following sets and numbers:
\begin{itemize}
\item $\Lambda_m(k):=\{\alpha\in \NN^m;\; |\alpha|=k\}$,
\item $\Delta_m(k):=\{\alpha\in \NN^m;\; |\alpha|\leq k\}$,
\item $L_m(k):=\#\Lambda_m(k)$ and $D_m(k):=\#\Delta_m(k)$ 
\end{itemize}
\end{defn}

Note that $L_m(k)$ (resp.\ $D_m(k)$) is the number of monomials of degree exactly (resp.\ at most) $k$ in $m$ variables.

Fix an integer $d$ and define, for all integers $n$ and $m$ such that $m<n$ the following integers.

\begin{equation}\label{eq: def mu r V e}
\begin{aligned}
\mu(n,d) &=D_n(d) \qquad& r(m,d)&=\min\{x\in \ZZ; \; D_m(x-1)\leq \mu < D_m(x)\}\\
V(n,d) &=\sum_{k=0}^d k L_n(k) \qquad & e(n,m,d) &=\sum_{k=1}^{r-1} kL_m(k)+r(\mu-D_m(r-1)).
\end{aligned}
\end{equation}

To apply the determinant method, we need the following lemma.

\begin{lem}\label{lem:determinant.estimate}
Let $K$ be a henselian field of equicharacteristic $0$. Let $t$ be a pseudo-uniformizer of $K$. Fix integers $\mu, r$ and $U$ an open subset of $K^m$ which is contained in a product of $m$ closed balls of radius $|t|^\rho$, where $\rho\geq 0$ is an integer. Fix $x_1, ..., x_\mu\in U$ and functions $\psi_1, ..., \psi_\mu: U\to K$. Assume that the $\psi_i$ satisfy $T_r$-approximation on $U$ for some integer $r$ with
\[
D_m(r-1)\leq \mu < D_m(r).
\]
Then 
\[
|\det (\psi_i(x_j))_{i,j})| \leq |t|^{\rho e}.
\]
\end{lem}
\begin{proof}
Straightforward adapation from \cite[Lem.\,3.3.1]{CCL-PW}.
\end{proof}

\section{Counting rational points on transcendental and algebraic curves} \label{sec:counting}

In this section we prove Theorems~\ref{thm:CC-c-dim}, \ref{thm:optimal}, and~\ref{thm:alg.curves}. Throughout, let $K$ be an equicharacteristic zero henselian valued field, equipped with a 1-h-minimal $\cL$-structure, for some language $\cL$ expanding the language of valued fields. We assume that $\acl=\dcl$ in $K$ and that the subgroup of $b$-th powers in the residue field $k^\times$ has finite index, for some $b>1$. Fix a lift $\tilde{k}$ of the residue field and fix a pseudo-uniformizer $t$.

\subsection{$T_r$-parametrizations}

Crucial to our approach is the following theorem, which asserts the existence of $T_r$-parametrizations for definable planar curves. 

\begin{thm}\label{thm:Tr.param.curves}
	Let $K$ be an equicharacteristic zero valued field, equipped with a 1-h-minimal $\cL$-structure. Assume that $\acl=\dcl$ in $K$ and that the subgroup of $b$-th powers of $k^\times$ has finite index, for some $b>1$. Let $Y\subset K^n$ be a definable set. Fix a positive integer $r$ and let $C\subset Y\times \cO_K^2$ be a definable set such that for every $y\in Y$, $C_y$ is a curve. Then there exist finitely many maps $\phi_1, \ldots, \phi_N: Y\times \cO_K\to C$ such that for every $y\in Y$, $\phi_{1, y}, \ldots, \phi_{N,y}$ form a $T_r$-parametrization for some subset $C'_y \subseteq C_y$, containing all points of $C_y \cap (\tilde{k}[t])^2$.
\end{thm}

To prove this theorem, we start with a $T_1$-parametrization of our curve, which exists because of Theorem \ref{thm:celldecomp.1.lipschitz}. To move from a $T_1$-parametrization to a $T_r$-parametrization for curves we will use power substitutions.

\begin{lem}\label{lem: T1 to Ts}
Let $X\subseteq K$ and let $f: X\to \cO_K$ be a $\emptyset$-definable $1$-Lipschitz function. Fix a positive integer $r$. Then there exists a finite $\emptyset$-definable set $C$ such that the following holds. Let $B$ be a ball $1$-next to $C$ contained in $\cO_K$, say $1$-next to $c\in C$. For $a,b\in \cO_K$ consider the map $p_r: K\to K: x\mapsto a(x-c)^r+b$. If $D$ is any open ball not containing $0$ with $p_r(D)\subseteq B$ then $f\circ p_r$ satisfies $T_r$-approximation on $D$ with respect to its Taylor polynomial. Moreover, for $y\in D$ and $j=1, ..., r$ there is the bound
\[
|\partial^j(f\circ p_r)(y)| \leq |y|^{r-j}.
\]
\end{lem}
\begin{proof}
Use Theorem \ref{thm: jacobian property} and Theorem \ref{thm: taylor approximation} to find a finite $\emptyset$-definable set $C$ such that $f$ satisfies Taylor approximation up to order $r$ on balls $1$-next to $C$ and such that the first $r$ derivatives of $f$ satisfy the Jacobian property on balls $1$-next to $C$. Without loss of generality, let $c=0$, $a=1$ and $b=0$ and fix $x_0\in D$. Then $x_0^r\in B$ and $\radop B = |x_0|^r$ since $B$ is $1$-next to $0$. The fact that $f$ is $1$-Lipschitz gives that
\[
|f'(x_0^r)|\leq 1. 
\]
By the Jacobian property, the first $r$ derivatives of $f$ all have a constant norm on $B$. Thus for $i\leq r$,
\[
f^{(i)}(B) \subseteq \{y\in K\mid |y| = |f^{(i)}(x_0^r)|\},
\]
and hence $\radop f^{(i)}(B)\leq |f^{(i)}(x_0^r)|$. On the other hand, the Jacobian property yields that
\[
|f^{(i)}(x_0^r)| = \frac{\radop f^{(i-1)}(B)}{\radop B} \leq \frac{|f^{(i-1)}(x_0^r)|}{|x_0^r|}.
\]
Using induction and the fact that $f$ is $1$-Lipschitz this gives for $1\leq i\leq r$ that
\begin{equation}\label{eq:bound.derivative}
|f^{(i)}(x_0^r)x_0^{r(i-1)}| \leq 1.
\end{equation}
Let $x\in D$. Then $x^r$ is in the same ball $1$-next to $0$ as $x_0^r$ and so $|x^r|=|x_0^r|$. Since both $f$ and $p_r$ have Taylor approximation up to order $r$ on their respective domains $B$ and $D$ we can conclude by \cite[Lem.\,3.2.7]{CCL-PW}.
\end{proof}

\begin{proof}[Proof of Theorem \ref{thm:Tr.param.curves}]
By enlarging $r$ if necessary, we may assume that $r$ is a power of $b$. Using the cell decomposition theorem \ref{thm:celldecomp.1.lipschitz} uniformly in $y\in Y$ we obtain for every $y\in Y$ finitely many sets $P_{i,y}$ whose union is $C_y$ such that every $P_{i,y}$ is, after a coordinate permutation, a $(1,0)$-cell or a $(0, 0)$-cell with $1$-Lipschitz centers. Since everything below works uniformly in $y$, we drop the subscript $y$ from now on. 

The $(0,0)$-cells are just singletons, so let us focus on one of the $(1,0)$-cells, say $P_\ell$. After a coordinate permutation we may assume that $P_\ell$ is the graph of a $1$-Lipschitz map
\[
\phi: P\subset \cO_K \to \cO_K,
\]
where $P$ is a cell with center $c \in K$. 
Since $P \subseteq \cO_K$, we may assume that $c \in \cO_K$. 
Additionally, we may assume that $\phi$ satisfies the conclusion of Lemma \ref{lem: T1 to Ts}.

Now suppose there exists some $p_0(t) \in \tilde{k}[t]$ such that $\abs{p_0(t) - c}$ is not an integer power of $\abs{t}$.
Since translation by $p_0(t)$ preserves $\tilde{k}[t]$, we may reduce to the case where $p_0(t) = 0$. In particular, for every rational point $p(t)$ it holds that $\abs{p(t)} \neq \abs{c}$.
Thus $P \cap \tilde{k}[t]$ is contained inside a cell $P' \subseteq P$ with center $0$.
Since we are only interested in the part of $C$ above $P'$, we may reduce to the case where $\abs{p(t) - c} \in t^{\ZZ}$ for all $p(t) \in \tilde{k}[t]$.

The group of $r$-th powers in $k^\times$ has finite index in $k^\times$, since $r$ is a power of $b$. Let $a_1, ..., a_m\in \cO_K^\times$ reduce to representatives for the cosets of $(k^\times)^r$ in $k^\times$. 
For $i=1, ..., m$ and $j=0, ..., r-1$ let
\[
D_{i,j} = \{y\in K\mid c +  a_i t^j y^r \in P\}.
\]
We claim that the finitely many maps
\[
p_{i,j}: D_{i,j}\to P : y \mapsto c + a_i t^j y^r
\]
cover $P \cap \tilde{k}[t]$. Indeed, by our previous reduction $\abs{p(t) -c} \in t^{\ZZ} $, and so we may write $p(t) = c + t^{n} \alpha$ for some $n \in \NN$ and $\alpha \in \cO_K^{\times}$. 

By Lemma \ref{lem: T1 to Ts} the $\phi\circ p_{i,j}$ are all $T_r$ on open balls contained in $D_{i,j}$. We prove that actually $\phi\circ p_{i,j}$ even has $T_r$-approximation on all of $D_{i,j}$. 
For convenience of notation, we assume $c =0$. The general case is similar.
So let $x,y\in D_{i,j}$. Since $D_{i,j}$ is a cell with center $0$, if $\rv(x) = \rv(y)$ then $x$ and $y$ are in the same ball contained in $D_{i,j}$ and we are done. So assume that $\rv(x)\neq \rv(y)$. Then
\begin{align*}
	\qquad|\phi(a_it^jx^r) - T_{\phi\circ p_{i,j}, y}^{<r}(x)| \leq \max\{|\phi(a_it^jx^r) - \phi(a_it^jy^r)|, |\phi(a_it^jy^r) - T_{\phi\circ p_{i,j}, y}^{<r}(x)|\}.
\end{align*}
For the first term, use that $\phi$ is $1$-Lipschitz to obtain
\[
|\phi(a_it^jx^r) - \phi(a_it^jy^r)| \leq |x^r - y^r| \leq \max\{|x|^r, |y|^r\} = |x-y|^r,
\]
since $\rv(x)\neq \rv(y)$. For the second term, we use the bound provided by Lemma \ref{lem: T1 to Ts} to get that
\begin{align*}
	|\phi(a_it^jy^r) - T_{\phi\circ p_{i,j}, y}^{<r}(x)| &\leq \max_{\ell=1, ..., r-1} \left|\frac{\partial^\ell(\phi\circ p_{i,j})(y)}{\ell!}(x-y)^\ell\right| \\
	&\leq \max_{\ell=1, ..., r-1}|y|^{r-\ell}|x-y|^\ell \leq |x-y|^r.
\end{align*}
So $\phi\circ p_{i,j}$ satisfies $T_r$-approximation on all of $D_{i,j}$. In conclusion, the maps 
\[
\psi_{ij}: D_{ij}\to C: y\mapsto (p_{ij}(y), \phi(p_{ij}(y))
\]
all have $T_r$-approximation. Their images cover $P_\ell \cap (\tilde{k}[t])^2$, since the images of the $p_{ij}$ cover $P \cap \tilde{k}[t]$.
\end{proof}

\subsection{Transcendental curves}

The following lemma is an adapted version of \cite[Lem.\,5.1.3]{CFL}. We use it to capture rational points of bounded height in a small ball in a single hypersurface.

\begin{lem}
\label{lem: Lemma 5.1.3 revisited}
Fix integers $d, m, n$ with $m<n$ and consider $r,V, e$ as defined in equation \ref{eq: def mu r V e}. Let $s$ be a positive integer, let $U\subseteq \cO_K^m$ and suppose $\psi=(\psi_1,...,  \psi_n): U\to \cO_K^n$ satisfies $T_r$-approximation. For $\alpha>\frac{sV}{e}$ a positive integer, denote by $p: \cO_K^m\to (\cO_K/(t^\alpha))^m$ the projection map. Then for any fibre $B$ of $p$, the image $\psi(B\cap U)_s$ is contained in an algebraic hypersurface of degree at most $d$. Moreover, $V/e$ goes to $0$ as $d$ goes to infinity.
\end{lem}
\begin{proof}
Let $B\subseteq \cO_K^m$ be a product of closed balls of radius $|t|^\alpha$, i.e.\ a fibre of the map $p$, and take points $P_1, ..., P_\mu$ in $\psi(B\cap U)_s$. Take $x_i\in B\cap U$ such that $\psi(x_i)=P_i$. Consider the determinant 
\[
\Delta = \det ((\psi(x_i))^j)_{1\leq i\leq \mu, j\in \Delta_n(d)}.
\] 
For $j\in \Delta_n(d)$, the notation $(y_1, ..., y_n)^j$ is to be interpreted as $\prod_i y_i^{j_i}$. Since $\psi$ satisfies $T_r$-approximation, Lemma \ref{lem:determinant.estimate} gives that $\ord_t(\Delta)\geq \alpha e$. Since the $P_i$ are in $\tilde{k}[t]^n_s$, if $\Delta$ were non-zero then $\ord_t (\Delta) \leq sV$. But $\alpha > \frac{sV}{e}$ so that $\Delta=0$. 

Now we use the determinant method. Since $\Delta=0$, the $\mu$ vectors $((\psi(x_i))^j)_j$ are linearly dependent. This implies that there exists some algebraic hypersurface of degree at most $d$ passing through all of the points $\psi(x_i)$. Since this holds for any $\mu$ points in $\psi(B\cap U)_s$ we can find such a hypersurface containing all of $\psi(B\cap U)_s$. The last fact follows from an easy explicit calculation, see e.g.\ \cite[p.\ 212]{P1}.
\end{proof}

We need one more projection lemma to reduce to the planar case. Let us call a set in $K^2$ \emph{non-algebraic up to degree $d$} if it has a finite intersection with every algebraic curve of degree at most $d$.

\begin{lem}\label{lem:projection}
Let $C\subset K^n$ be a definable transcendental curve and fix a positive integer $d$. Then there exists a finite definable partition of $C$ into sets $C_i$, together with coordinate projections $\pi_i: C_i\to K^2$ such that $\pi_i|_{C_i}$ is a bijection onto its image, and $\pi_i(C_i)$ is non-algebraic up to degree $d$.
\end{lem}
\begin{proof}
By the cell decomposition theorem \ref{thm:celldecomp.1.lipschitz}, we may partition $C$ into finitely many definable sets $C_i$ such that after a coordinate permutation, $C_i$ is a $(1, 0, ..., 0)$ cell with $1$-Lipschitz centers. (We may disregard the finitely many $(0, ..., 0)$-cells since there are just points.) In other words, $C_i$ is the graph of a map
\[
\phi: P\to K^{n-1}: x\mapsto (\phi_1(x), ..., \phi_{n-1}(x)),
\]
where all $\phi_j$ are $1$-Lipschitz, and $P\subset K$ is a $1$-cell. Denote by $A_j\subset P$ the set of $x\in P$ such that in some neighbourhood of $x$, the graph of $\phi_j$ is non-algebraic up to degree $d$. Then the $A_j$ are definable sets, and they cover $P$, since else $\phi$ would not be transcendental. Thus we can further partition $C_i$ into the graphs of $\phi$ over every $A_j$. Over $A_j$, projection onto the first and $j$-th coordinate gives the desired conclusion.
\end{proof}

\begin{remark}
The use of this lemma can be avoided by working with cylinders over $X$ for each possible projection $K^n\to K^2$, see \cite[p.\,45]{CCL-PW}.
\end{remark}

With this, we can prove our main result. Recall that the strategy of the proof is as follows. Using $T_r$-parametrizations, we represent $C$ as a finite union of graphs of functions satisfying $T_r$-approximation. Then Lemma \ref{lem: Lemma 5.1.3 revisited} gives a hypersurface catching all rational points on $C$ of height at most $s$. The fact that $C$ is transcendental then gives the desired conclusion.

\begin{proof}[Proof of Theorem \ref{thm:CC-c-dim}]
Recall the definition of $V, e$ and $r$ from equation~\ref{eq: def mu r V e} and recall that $V/e$ goes to zero as $d$ goes to infinity, by Lemma~\ref{lem: Lemma 5.1.3 revisited}. Take $d$ such that
\[
\frac{V}{e} < \varepsilon.
\]
Up to enlarging $r$ if necessary, we may assume that $r$ is a power of $b$. By applying Lemma \ref{lem:projection} we may assume without loss of generality that $C$ is a planar curve in $\cO_K^2$ which is non-algebraic up to degree $d$.

By Theorem \ref{thm:Tr.param.curves} there exist finitely many maps $\psi_i: U_i\subset \cO_K\to C$ which together form a $T_r$-parametrization of a subset $C' \subset C$ which already contains all rational points of $C$. Let us focus on one such $\psi_i$. By our construction, we may apply Lemma \ref{lem: Lemma 5.1.3 revisited} with $\alpha = \lceil s\varepsilon \rceil$. This yields that $\psi_{i}(U_i\cap B)_s$ is contained in an algebraic hypersurface $X$ of degree at most $d$, for any closed ball $B\subset \cO_K$ of radius $|t|^\alpha$. Since $\psi_{i}(U_i)$ is contained in $C$, and since $C$ is non-algebraic up to degree $d$, this intersection $X\cap C$ is finite. Even more, by uniform finiteness in definable families~\cite[Lem.\,2.5.2]{CHR}, the intersection of $X$ with $C$ is uniformly bounded (over all such $X$) by some integer $N$. Thus the counting dimension of $C$ is bounded by 
\[
(NN', 1, \lceil \varepsilon s \rceil),
\]
where $N'$ is the total number of sets $U_i$ required in the cell decomposition for $C$.

Finally, to prove a uniform upper bound on counting dimension in definable families, assume that $(C_y)_y$ is a definable family of transcendental curves, for $y\in Y\subseteq K^m$ with $Y$ definable. Then the above proof can easily be made uniform in $C_y$. Indeed, by Lemma \ref{lem:projection} we can assume that every $C_y$ is a planar curve which is non-algebraic up to degree $d$. The number of maps for a $T_r$-parametrization can be uniformly bounded in $y$ since Theorem \ref{thm:Tr.param.curves} is uniform in families. Similarly, the intersection of an algebraic curve of degree at most $d$ with any $C_y$ is finite, and thus uniformly bounded by \cite[Lem.\,2.5.2]{CHR}. These two facts give the desired conclusion. 	 
\end{proof}
\subsection{Linear upper bounds are optimal} 

%
We now prove Theorem \ref{thm:optimal}. This shows that the bound in Theorem \ref{thm:CC-c-dim} is optimal, in the sense that one cannot replace the last component $\lceil \varepsilon \cdot s \rceil$ by a sublinear function $e(s)$, even if we allow $N(s)$ to be completely arbitrary. 

We first recall the notion of rings of strictly convergent power series $\cO_K\langle x_1, \dots, x_n \rangle$. By definition, their elements are those power series $\sum_{i \in \NN} a_i x^i$ with coefficients in $\cO_K$ such that $a_i \to 0$, when $\abs{i} \to \infty$. Each $f \in \cO_K \langle x_1, \dots, x_n\rangle$ can be naturally considered as a function $\cO_K^n \to K$ and then also as a function $f \colon K^n \to K$, after extending by zero. By \cite[Thm.\,6.2.1]{CHR} there exists a $1$-h-minimal structure on $K$ in which these functions are definable. Moreover, by \cite[Prop.\,4.3.3]{CHR}, there exists such a structure in which $\acl = \dcl$.

\begin{proof}[Proof of Theorem \ref{thm:optimal}]
We work in the structure on $K \coloneqq k((t))$ as outlined above, in which all functions $\cO_K^n\to \cO_K$ defined by a strictly convergent power series are definable. 

First fix any strictly increasing continuous function $\delta \colon \RR_{\geq 0} \to \RR_{>0}$ such that $e(s) \leq \delta(s)$ for all $s \in \NN$, with $\lim_{s \to \infty} \delta(s) = + \infty$ and $\lim_{s \to +\infty } \delta(s)/s = 0$.
Then choose any strictly increasing function $F \colon \RR_{\geq0} \to \RR_{>0}$ with $F(\NN) \subseteq \NN$ and such that for all $s \in \NN$
\[
F(\delta(s)) > N(s).
\]
Next, take any strictly increasing sequence of natural numbers $(N_n)_n$, with the property that for all $n \in \NN$
\[
3 N_{n-1}^2 F(N_{n-1}) < \frac{\delta^{-1}(N_n) - 1}{N_n}.
\]
Such a sequence exists since $\lim_{u\to \infty} \delta^{-1}(u)/u = \infty$ as $\delta$ is sublinear.
From these data, we construct $f \in \cO_K\langle x \rangle$ as
\[
f(x) = \sum_{n = 0}^{\infty} t^{N_n} x^{N_n} \prod_{i,\ell = 1}^{N_n} \prod_{j = 1}^{F(N_n)} (x - i - j t^\ell). 
\]
Letting the sequence $(N_n)_n$ grow even faster, if necessary, we may assume that for each $d \in \NN$ there is some $n_d$ such that for all $n \geq n_d$ we have $N_{n} > d(N_{n-1} + N_{n-1}^2 F(N_{n-1}) )$. Hence, for $M = N_{n-1} + N_{n-1}^2 F(N_{n-1})$,  the order of contact between $f$ and its $M$-th order Taylor approximation exceeds $d  M$. 
B\'ezout's theorem thus implies that $f$ cannot be algebraic of any degree $d \in \NN$ (as in \cite[Prop.\,1]{BCN}).
Hence the graph $C$ of $f \colon \cO_K \to K$ is a definable transcendental curve in $K^2$. 

We now show that the counting dimension of $C$ is not bounded by $(N, 1, e)$. So let $g: K^2\to \cO_K$ be any definable map. We will show that the composition
\[
C_s  \xrightarrow{g} \cO_K \xrightarrow{\operatorname{proj}} \frac{\cO_K}{ (t^{e(s)}) }
\]
has a fibre of size strictly larger than $N(s)$ for some sufficiently large $s$.

Take any $n$ and let $s$ be such that
\[ 
\delta(s) \leq N_n < \delta(s+1).
\]
For each $i,j,\ell \in \NN$, we have by construction that $f(i + jt^\ell) \in k[t]$. By our choice of $s$, we moreover have that $t$-degree of $f(N_n + j t^{N_n})$ is strictly smaller than $s$, when $1 \leq j \leq F(N_n)$.
Indeed, it follows from the construction of $(N_n)_n$ that
\begin{align*}
	\deg_t(f(N_n + j t^{N_n})) & \leq N_{n-1} +  N_n N_{n -1} + N_n N_{n-1}^2 F(N_{n -1})  \\
	&\leq 3 N_n  N_{n-1}^2 F(N_{n-1}) \\
	&< 	\delta^{-1}(N_n) - 1 \\
	&< s.
\end{align*}%
Now define 
\[
S  \coloneqq \{  (N_n + j t^{N_n},f(N_n + j t^{N_n})) \mid 1 \leq j \leq F(N_n)\}
\]
and note that $S \subseteq C_s$ by the above computation.

Define $h \colon \cO_K \to \cO_K \colon x \mapsto g(x,f(x))$. 
By Corollary \ref{cor:jac} we may assume that, possibly after increasing $n$ (and $s$),
\[
\abs{h(N_n + jt^{N_n}) - h(N_n))} \leq \abs{t^{N_n}}
\]
for all $j \in \NN$.
As $N_n \geq \delta(s) \geq e(s)$, this implies that all points of $S$ belong to the fibre at $N_n$ of the composition
\[ 
C_s  \xrightarrow{g} \cO_K \xrightarrow{\operatorname{proj}} \frac{\cO_K}{ (t^{e(s)}) }. 
\]
But by construction of $F$, the inequality $N_n \geq \delta(s)$ also implies $F(N_n) > N(s)$.	
As $\# S = F(N_n)$, we have found a fibre containing more than $N(s)$ elements of $C_s$.
\end{proof}


\subsection{Algebraic curves}\label{sec:algebraic.curves}

In this section we prove Theorem \ref{thm:alg.curves}, following along the lines of \cite[Sec.\,5]{CCL-PW}. We need some results on Hilbert functions.

For $r$ a positive integer, denote by $K[x_0, ..., x_n]_r$ the space of homogeneous degree $r$ polynomials. For $I$ a homogeneous ideal in $K[x_0, ..., x_n]$, let $I_r = K[x_0, ..., x_n]_r\cap I$ and denote by $H_I(r) = \dim K[x_0, ..., x_n]_r / I_r$ the Hilbert function of $I$. Let $<$ be the monomial order on $K[x_0, ..., x_n]$ defined by $x^\alpha < x^\beta$ if $|\alpha|<|\beta|$ or $|\alpha|=|\beta|$ and $\alpha_i > \beta_i$ for some $i$ and $\alpha_j = \beta_j$ for $j<i$. After reordering the variables, this is the graded reverse lexicographic order on monomials. Denote by $\LT(I)$ the ideal of leading terms of $I$, where the leading term of a homogeneous element $p(x)$ of $K[x_0, ..., x_n]$ is the monomial in $p(x)$ which is maximal for $<$. Then $I$ and $\LT(I)$ have the same Hilbert functions, by \cite[Ch.\,9 Prop.\, 3.9]{CoxLittleShea}. For $i\in \{0, ..., n\}$ define
\[
\sigma_{I,i}(r) = \sum_{|\alpha| = r, x^\alpha\notin \LT(I)} \alpha_i
\]
and note that $rH_I(r) = \sum_i \sigma_{I,i}(r)$. Let $X$ be an irreducible variety in $\PP^n_K$ of degree $d$ and dimension $m$, with homogeneous ideal $I$. The Hilbert function $H_I(r)$ of $I$ agrees with the Hilbert polynomial $P_X(r)$ of $X$, for $r$ sufficiently large. Recall that this is a degree $m$ polynomial whose leading coefficient is $d/m!$. By \cite{Broberg}, for $i=0, ..., n$ there are real numbers $a_{I,i}\geq 0$ such that
\[
\frac{\sigma_{I,i}(r)}{rH_I(r)} = a_{I,i} + O_{n,d}(1/r), \quad \text{ for }r\to \infty.
\]
Note also that $a_{I,0} + ... + a_{I,n} = 1$. We can now prove Theorem \ref{thm:alg.curves}.

\begin{proof}[Proof of Theorem \ref{thm:alg.curves}]
We have an irreducible algebraic curve $C$ in $\AA_K^2$ of degree $d$. Put $C' = C(K)\cap \cO_K^2$. Consider the embedding
\[
\iota: \AA_K^2\to \PP_K^2: (x,y) \mapsto (1:x:y)
\]
and let $I$ be the homogeneous ideal of the closure of $\iota(C)$ in $\PP_K^2$. Let $\delta$ be a positive integer, which we will choose later depending on $s$, and define
\[
M(\delta) = \{ j\in \NN^3\mid |j| = \delta, x^j\notin \LT(I)\}.
\]
Let $\mu = \# M(\delta) = H_I(\delta), \sigma_i = \sigma_{I,i}(\delta)$ for $i=0,1,2$ and put $e = \mu(\mu-1)/2$. By Theorem \ref{thm:Tr.param.curves} there exist finitely many maps $\phi_1, ..., \phi_N: Y_i\subset \cO_K \to C'$ forming a $T_r$-parametrization of $C'$. Let $B_{\alpha}$ be a closed ball in $\cO_K$ of radius $|t|^\alpha$ for some integer $\alpha$. Fix a positive integer $s$, take points $y_1, ..., y_\mu$ in $(\phi_i(B_\alpha\cap Y_i))_s$ and consider the determinant
\[
\Delta = \det\left( \iota(y_i)^j \right)_{j\in M(\delta), 1\leq i\leq \mu}.
\]
By Lemma \ref{lem:determinant.estimate} we have that $|\Delta | \leq |t|^{\alpha e}$. On the other hand, since $\iota(y_i)$ has coordinates which are polynomials of degree $<s$, we find that $\Delta$ is in $\tilde{k}[t]$ of degree
\[
\deg_t \Delta \leq (s-1)(\sigma_1 + \sigma_2).
\]
Therefore, if we take $\alpha>(s-1)(\sigma_1+\sigma_2)/e$, then $\Delta = 0$. As in the proof of Lemma \ref{lem: Lemma 5.1.3 revisited}, using the determinant method, we can find a polynomial $H$ in two variables with coefficients in $\tilde{k}[t]$, and exponents in $M(\delta)$ which vanishes on $(\phi_i(B \cap Y_i))_s$. Since the exponents of $H$ lie in $M(\delta)$, we also see that $H$ does not vanish identically on $C$. By B\'ezout's theorem, the intersection of $H=0$ and $C$ consists of at most $\delta d$ points.

We want to conclude by taking $\alpha = \lceil s/d \rceil$, so we look for a suitable $\delta$ now. Similarly to the proof of \cite[Thm.\,5.1.3]{CCL-PW}, one obtains that
\[
\sigma_i / e = 2\alpha_i/d + O_d(\delta^{-1}).
\]
By a result of Salberger \cite[Lem.\,1.12]{SalbCrelle}, it follows that
\[
\frac{\sigma_1 + \sigma_2}{e} \leq \frac{1}{d} + O_d(\delta^{-1}).
\]
Hence we may take $\delta = sO_d(1)$ so that
\[
\frac{(s-1)(\sigma_1+\sigma_2)}{e} < \lceil s/d\rceil = \alpha.
\]
By Proposition \ref{prop:basic.props} we can find a definable map $f: C'\to \cO_K$ such that the composition
\[
C_s \to \cO_K \to \frac{\cO_K}{(t^{\lceil s/d\rceil})}
\]
has finite fibres of size at most $Nd\delta = NsO_d(1)$, for every $s$. Now, the number of cells $N$ required in the $T_r$-parametrization of $C$ can be made uniform in definable families. Since the set of degree $d$ curves in $K^2$ is a definable family, we may assume that $N = O_d(1)$. So we conclude that the counting dimension of $C$ is bounded by 
\[
(sO_d(1), 1, \lceil s/d\rceil). \qedhere
\]
\end{proof}

\begin{example} \label{ex:alg_optimal}
We show that one cannot improve the last component in the counting dimension for algebraic curves. Let $K$ be any equicharacteristic zero valued field equipped with a 1-h-minimal structure. Fix a pseudo-uniformizer $t$ and a lift $\tilde{k}$ of the residue field. Denote by $C$ the curve in $\cO_K^2$ defined by $y=x^d$. We claim that the counting dimension of $C$ is not bounded by $(N(s), 1, e(s))$ for any functions $N, e: \NN\to \NN$ for which $e(s) < \lceil s/d \rceil$ when $s$ is sufficiently large. Let $f: C\to \cO_K$ be any definable map.


Take $s' = sd+1$ sufficiently large such that $e(s') < \lceil s'/d\rceil = s+1$. In particular note that $e(s') \leq s$. Define the map $g: \cO_K\to C: x\mapsto (x,x^d)$ and put $h = f\circ g: \cO_K\to \cO_K$. By Corollary \ref{cor:jac}, there exists an $a \in \tilde{k}$,  such that for any $b\in \tilde{k}$ we have that
\[
|h(a) - h(a + bt^s)| \leq |t^s|.
\]
This implies that $h(a)\equiv h(a+bt^s) $ in $\cO_K / (t^{e(s')})$, for all $b \in \tilde{k}$. Finally, by noting that $g(a + bt^s)$ lies in $C_{s'}$, one sees that the map
\[
C_{s'} \xrightarrow{f} \cO_K \xrightarrow{\operatorname{proj}} \frac{\cO_K}{ (t^{e(s')}) }
\]
has an infinite fibre. 
\end{example}

\section{Curves with uniformly bounded counting dimension} \label{sec:QQpt}
We now consider an analytic structure on $\QQ_p((t))$ where each $\emptyset$-definable curve $C$ has an associated constant $N_C \in \NN$ such that $\cdim(C_s) \leq (N_C,1,1)$. Contrast this with Theorem~\ref{thm:optimal}, where we produced curves, definable in some analytic structure, whose counting dimension can not be bounded by any constant triple $(N,1,1)$.
The stronger upper bounds in this section result from working in a more restricted analytic structure. The essential difference with the setting of Theorem~\ref{thm:optimal} is that we now only add function symbols for power series whose coefficients do not involve the uniformizer $t$.

%
%

Throughout this section we will take $t \in \QQ_p((t))$ as our chosen pseudo-uniformizer and $\tilde{k} = \QQ_p \subseteq \QQ_p((t))$ as our lift of the residue field.
\begin{defn}
Let $\LZpan$ be the language expanding the language of valued fields $\Lval = \{0,1,+,\cdot,\cO_K\}$ by
\begin{enumerate}
	\item a binary function symbol ``$-$'' and a unary function symbol $(\cdot)^{-1}$,
	\item a unary relation symbol $\cO_{K,\fine}$,
	\item $n$-ary function symbols for the elements of the rings of strictly convergent power series $\ZZ_p\langle x_1,\dots,x_n \rangle$, for $n \in \NN$.
\end{enumerate}
\end{defn}

The field $K = \QQ_p((t))$ admits a natural $\LZpan$-structure. We interpret $\cO_{K,\fine}$ as the valuation ring $\ZZ_p + t\QQ_p[[t]]$ and $\cO_K$ as its equicharacteristic zero coarsening $\QQ_p[[t]]$. Each function symbol $f \in \ZZ_p\langle x_1,\dots,x_n\rangle$ is interpreted naturally as a function $\cO_{K,\fine}^n \to K$. Finally, ``$-$'' and $(\cdot)^{-1}$ are just subtraction and inversion on $K$, where the latter is extended by $0^{-1} = 0$. Note that $\QQ_p((t))$, equipped with the valuation ring $\cO_K$, is $1$-h-minimal for this structure by \cite[Thm. 6.2.1]{CHR}.
\begin{theorem} \label{thm:uniform_bound}
For each transcendental curve $C \subseteq \QQ_p((t))^n$, $\emptyset$-definable in $\LZpan$, there is some constant $N_C \in \NN$ such that
\[ 
\cdim(C_s) \leq (N_C,1,1) .
\]
\end{theorem}
\begin{remark}
As before, the constant $N_C$ can be made uniform in definable families.
\end{remark}
%
%
%
%
\begin{lem} \label{lem:l-isom}
For all $\lambda \in \ZZ_p^{\times}$ the map
\[ 
\tau_{\lambda} \colon \QQ_p((t)) \to \QQ_p((t)) \colon \sum_{j = k}^{+ \infty} a_j t^j \mapsto \sum_{j = k}^{+ \infty} a_j (\lambda t)^j 
\]
is an $\LZpan$-automorphism.
\end{lem}
\begin{proof}
Fix some $\lambda \in \ZZ_p^{\times}$.
It is clear that $\tau_{\lambda}$ fixes the constants $0,1$ and that the relations
$ x \in \QQ_p[[t]] $ and $x \in \ZZ_p + t \QQ_p[[t]]$ are invariant under $\tau_{\lambda}$.
Similarly, it is straightforward to verify that $\tau_{\lambda}$ commutes with addition, multiplication and inversion.

It thus remains to check that $\tau_{\lambda}$ commutes with all function symbols $f$ in $\ZZ_p\langle x_1,\dots,x_n \rangle$. 
Take $(x_i)_{i=1}^n = (\saij)_{i=1}^n \in \QQ_p((t))^n$.
Then $f(x)$ is computed as follows.
\[ 
f(x) = \begin{cases}
	0 & \text{ if } x_i \notin \ZZ_p + t\QQ_p[[t]] \text{ for some } i, \\
	\sum_{s = 0}^{+ \infty} f_s((a_{ij})_{i,j}) t^s & \text{ else},
\end{cases} 
\]
where each $f_s$ is some quasi-homogeneous polynomial of degree $s$ in variables $a_{ij}$, each of weight $j$, for $1 \leq i \leq n$ and $0 \leq j \leq s$.
This precisely means that $f_s(( \lambda^j a_{ij})_{i,j})  = f_s( (a_{i,j})_{i,j}  ) \lambda^s $. It follows that $f$ commutes with $\tau_{\lambda}$.
\end{proof}
We will need that transcendental curves only have finite intersection with any semialgebraic curve: one-dimensional subsets of $K^n$ definable in $\Lval \cup K$.
This follows from the fact that semialgebraic curves are locally algebraic, as made precise by the lemma below.
\begin{lem} \label{lem:transcendence_defs_equivalent}
Let $K$ be a valued field with valuation ring $\cO_K$.
Let $C$ be a transcendental curve in $K^n$.
Then, for any $K$-definable curve $X$ in $\Lval$, the intersection $C \cap X$ is finite.
\end{lem}
\begin{proof}
By valued field-quantifier elimination for henselian valued fields \cite[Prop.\,4.3]{Flen}, $C$ is a finite union of sets $V \cap S$, where $V$ is the vanishing locus of some polynomials with coefficients in $K$ and $S$ is of the form $(\rv(f_i(x)))_{i =0}^n \in R$ for certain $f_i \in K[x]$ and $R \subseteq (\RV^{\times})^n$. As $X$ is of dimension one and $S$ is open, only the zero- and one-dimensional irreducible components of $V$ can have non-empty intersection with $S$.
In particular, if $V \cap S$ met $C$ at infinitely many points, then there would be an algebraic curve $X' \subseteq V$ containing infinitely many points of $C$.
%
\end{proof}
We now continue along the same lines as in the proof of \cite[Thm.\,1]{BCN}, using Lemma \ref{lem:l-isom} instead of the quantifier elimination statement used there.
\begin{prop} \label{prop:X1_catches_all}
Let $C \subseteq \QQ_p((t))^n$ be a transcendental curve which is $\emptyset$-definable in $\LZpan$.
Then $C_s \subseteq C_1$ for all integers $s >0 $. In particular, $\cdim(C_s) \leq (1,n,1)$.
\end{prop}
\begin{proof}
Suppose $(\sum_{j = 0}^{s-1} a_{ij} t^j)_i \in C_{s} \setminus C_1$.
Consider $C_{s}$ as a subset $A \subseteq \QQ_p^{n s}$ via the identification 
\[
\left(\sum_{j = 0}^{s-1} b_{ij} t^j\right)_i \mapsto (b_{ij})_{i,j}.
\]
Let $X$ be an algebraic curve in $\QQ_p^{n s}$ containing all points $( \lambda^{j} a_{ij})_{i,j}$ for $\lambda \in \QQ_p$.
By the above Lemma \ref{lem:l-isom}, it follows that $X \cap A$ contains all points $( \lambda^{j} a_{ij})$ for $\lambda \in \ZZ_p^{\times}$. In particular, it is infinite.

Let $Y$ be the image of $X(\QQ_p[[t]])$ under $(x_{ij})_{i,j} \mapsto (\sum_{j =0}^{s-1} x_{ij} t^j)_i$.
As this map and the curve $X(\QQ_p[[t]])$ are definable in the $1$-h-minimal $\cL_{\val}$-structure on $\QQ_p((t))$ (for the valuation ring $\QQ_p[[t]]$), it follows from \cite[Prop.\,5.2.4 (3,4)]{CHR} that $Y$ is an $\Lval \cup \QQ_p((t))$-definable set of dimension at most $1$.
Since the infinite set $X$ injects into $Y$, it then follows by \cite[Prop.~5.2.4 (1)]{CHR} that $Y$ has dimension exactly $1$.
Now use that $X \cap A$ is infinite, whence so is $Y \cap C_{s}$. By Lemma \ref{lem:transcendence_defs_equivalent}, this contradicts the assumption that $C$ is transcendental.
\end{proof}
%
%
\begin{proof}[Proof of Theorem \ref{thm:uniform_bound}]
By \cite[Thm. 5.7.3]{CHR} we may assume that $C$ is a single reparametrized cell $(A,\sigma)$.
As $C$ is one-dimensional, it follows that either $A$ is a finite collection of points (in which case we are done) or $A$ is of type $(1,0,\dots,0)$, up to a coordinate permutation of $K^n$.
Thus $A$ is the coordinate projection onto $K^n$ of the graph of some $\LZpan$-definable function $c \colon P \subseteq K \times \RV^{\ell} \to K^{n-1}$ (for some $\ell \in \NN$).
By \cite[Cor. 2.6.7]{CHR} and \cite[Lem. 2.5.2]{CHR} it holds that $\#{c(x,\RV^{\ell})} \leq N$ for some $N \in \NN$, independent of $x$.

Now consider the map $g \colon K^n \to K$, wich is the projection onto the first coordinate on $\cO_K^n$ and is identically zero on $K^n \setminus \cO_K^n$.
For any $x \in \cO_K/(t) \cong \tilde{k} \subseteq \cO_K$, the fibre at $x$ of the composition
\[ A_1  \xrightarrow{g} \cO_K \xrightarrow{\operatorname{proj}} \frac{\cO_K}{(t)} \]
is precisely $(\{x\} \times c(x,\RV^{\ell}) ) \cap A_1$. In particular, it has size at most $N$.
As $A_s = A_1$ for all $s \in \NN$, by proposition \ref{prop:X1_catches_all}, it follows that $\cdim(A) \leq (N,1,1)$.
\end{proof}

	\bibliographystyle{amsplain}
	\bibliography{anbib}
\end{document}